\documentclass[12pt]{amsart}

\usepackage{amssymb}

\usepackage{url}
\usepackage[all]{xy}

\usepackage[colorlinks=true,linkcolor=blue,urlcolor=blue,citecolor=blue,pagebackref]{hyperref} 

\usepackage[usenames, dvipsnames, svgnames]{xcolor}

\usepackage{enumerate} 

\usepackage{helvet} 
\usepackage{courier} 
\usepackage{eucal} 
\usepackage{mathrsfs} 

\reversemarginpar 
\normalmarginpar 

\let\oldmarginpar\marginpar 
\renewcommand\marginpar[1]{\-\oldmarginpar{\raggedright\small\sf #1}}

\title{Trianguline lifts of global mod $p$ Galois representations}

\newcommand{\nc}{\newcommand}

\nc{\rnc}{\renewcommand}

\nc{\bs}{\backslash}
\nc{\te}{\otimes}
\nc{\lf}{\lfloor} 
\nc{\rf}{\rfloor}
\nc{\lc}{\lceil}  
\nc{\rc}{\rceil}
\nc{\lr}{\longrightarrow}
\nc{\sr}{\stackrel}
\nc{\dar}{\dashrightarrow}
\nc{\thra}{\twoheadrightarrow}
\nc{\hra}{\hookrightarrow}

\nc{\la}{\langle}
\nc{\ra}{\rangle} 

\nc{\ms}{\mathscr}
\nc{\mc}{\mathcal}
\nc{\mb}{\mathbb}
\nc{\mf}{\mathbf}
\nc{\mr}{\mathrm}
\nc{\mg}{\mathfrak}

\nc{\bP}{\mathbb{P}}
\rnc{\P}{\mathbb{P}}
\nc{\Q}{\mathbb{Q}}
\nc{\Z}{\mathbb{Z}}
\nc{\C}{\mathbb{C}}
\nc{\R}{\mathbb{R}}
\nc{\A}{\mathbb{A}}
\nc{\V}{\mathbb{V}}
\nc{\W}{\mathbb{W}}
\nc{\N}{\mathbb{N}}
\nc{\D}{\mathbb{D}}
\nc{\G}{\mathbb{G}}
\nc{\F}{\mathbb{F}}
\nc{\qb}{\overline{\mathbb{Q}}}
\nc{\qpb}{\overline{\mathbb{Q}}_p}

\nc{\del}{\partial}
\nc{\sq}{\square}

\nc{\wt}{\widetilde}
\nc{\wh}{\widehat}
\nc{\ov}{\overline}
\nc{\un}{\underline}

\nc{\aff}{{\A}^1}
\nc{\naive}{\!\sim_n}

\nc{\omx}{\omega_X}

\nc{\ep}{\epsilon}
\nc{\ve}{\varepsilon}
\nc{\vt}{\vartheta}
\nc{\vpi}{\varpi}

\rnc{\l}{\lambda}
\rnc{\k}{\kappa}

\rnc{\sl}{\shoveleft}

\nc{\res}{\operatorname{Res}}
\nc{\pic}{\operatorname{Pic}}
\nc{\spec}{\operatorname{Spec}}
\nc{\im}{\operatorname{Im}}
\nc{\gal}{\operatorname{Gal}}
\nc{\fr}{\operatorname{Fr}}
\nc{\ed}{\operatorname{ed}}
\nc{\rank}{\operatorname{rank}}
\nc{\h}{\operatorname{H}}
\nc{\ch}{\operatorname{char}}
\nc{\sw}{\operatorname{sw}}
\nc{\rsw}{\operatorname{rsw}}
\nc{\supp}{\operatorname{supp}}
\rnc{\hom}{\operatorname{Hom}}
\nc{\spf}{\operatorname{Spf}}

\nc{\Mor}{\operatorname{Mor}}
\nc{\Per}{\operatorname{Per}}
\nc{\prep}{\operatorname{Prep}}
\nc{\End}{\operatorname{End}}
\nc{\Orb}{\operatorname{Orb}}

\nc{\br}{\bar{\rho}}
\rnc{\un}{\mr{univ}}
\nc{\der}{\mr{der}}
\nc{\reg}{\mr{reg}}
\nc{\tri}{\mr{tri}}
\nc{\fgder}{\mg{g}^{\der}}
\nc{\red}{\mr{red}}

\nc{\rep}{\mr{Rep}}

\nc{\rb}{\bar{r}}

\nc{\co}{\mc{O}}

\nc{\Ad}{\mr{Ad}}

\nc{\glo}{\text{$\mr{GL}_n$-odd}}

\nc{\sln}{\mg{sl}_n}

\nc{\gln}{\mr{GL}_n}

\nc{\fg}{(\phi,\Gamma)}
\nc{\gfg}{(G,\phi,\Gamma)}
\nc{\tfg}{(T, \phi, \Gamma)}
\nc{\bfg}{(B, \phi, \Gamma)}

\nc{\ro}{\mathcal{R}}

\newtheorem{thm}{Theorem}[section]
\newtheorem{prop}[thm]{Proposition}

\newtheorem{cor}[thm]{Corollary}
\newtheorem{lem}[thm]{Lemma}

\theoremstyle{definition}
\newtheorem{defn}[thm]{Definition}

\newtheorem{rem}[thm]{Remark}

\newtheorem*{assA}{Assumption A}
\newtheorem*{assB}{Assumption B}

\numberwithin{equation}{section}

\begin{document}
\author[N.~Fakhruddin]{Najmuddin Fakhruddin}
\address{School of Mathematics, Tata Institute of Fundamental Research, Homi Bhabha Road, Mumbai 400005, INDIA}
\email{naf@math.tifr.res.in}
\author[C.~Khare]{Chandrashekhar Khare}
\address{UCLA Department of Mathematics, Box 951555, Los Angeles, CA 90095, USA}
\email{shekhar@math.ucla.edu}
\author[S.~Patrikis]{Stefan Patrikis}
\address{Department of Mathematics, Ohio State University, 100 Math Tower, 231 West 18th Ave., Columbus, OH 43210, USA}
\email{patrikis.1@osu.edu}

\maketitle

\begin{abstract}
  We show that under a suitable oddness condition, for $p \gg_F n$
  irreducible $n$-dimensional mod $p$ representations of the absolute
  Galois group of an arbitrary number field $F$ have characteristic
  zero lifts which are unramified outside a finite set of primes and
  trianguline at all primes of $F$ dividing $p$. We also prove a
  variant of this result under some extra hypotheses for
  representations into connected reductive groups.
\end{abstract}

\section{Introduction}

For any field $F$ we let $\Gamma_F$ be its absolute Galois group. In
the articles \cite{fkp-duke} and \cite{fkp-invent} we constructed
geometric (in the sense of Fontaine--Mazur) lifts for odd
representations $\br: \Gamma_F \to G(k)$, where $F$ is a totally real
number field, and $G$ is a split reductive group (possibly
disconnected) over the ring of integers $\co$ of a $p$-adic field $E$
with residue field $k$, when $\br$ satisfies certain additional
conditions. For example, if $G$ is connected,
$\br|_{\Gamma_{F(\zeta_p)}}$ is absolutely irreducible, and lifts
exist locally (which are regular de Rham at primes dividing $p$) at
all primes of $F$, then global lifts exist whenever $p$ is
sufficiently large by \cite[Theorem A]{fkp-duke}; in \cite{fkp-invent}
we constructed lifts for reducible representations under some more
technical hypotheses. The oddness assumption is crucial, but if
$G = \gln$ and $n>2$ then no $\br$ is odd in the sense of
\cite[Definition 1.2]{fkp-duke}, so the results of \cite{fkp-duke} and
\cite{fkp-invent} cannot be used to construct geometric
lifts.\footnote{The results of these papers do apply---and have
  interesting consequences---when $G = \gln$ and $F$ is a global
  function field, and even for number fields if we do not impose any
  conditions at primes dividing $p$.}  In fact, Calegari has proved
\cite[Theorem 5.1]{calegari:even2} that without any oddness condition
whatsoever geometric lifts need not exist even when $n=2$ and $F= \Q$.
The goal of this note is to show that if we weaken the requirement
that the lift be geometric to being unramified outside a finite set of
primes and trianguline\footnote{The reader may consult
  \cite{berger-triang} for a survey of trianguline representations; we
  give a brief sketch of the definition in \S \ref{s:bhs}.}  at all
primes above $p$, then the methods of \cite{fkp-duke} and
\cite{fkp-invent} can be adapted to apply even when $G = \gln$ and $F$
is an arbitrary number field if we assume the following weaker oddness
condition:
\begin{defn} \label{d:odd} %
  $ $
  \begin{enumerate}
  \item An involution $c$ in $\gln(K)$, where $K$ is a field of
    characteristic $\neq 2$, is said to be \emph{$\glo$} if
    $|n^+(c) - n^-(c)| \leq 1$, where $n^+(c)$ (resp.~$n^-(c)$) is the
    number of eigenvalues of $\rho(c)$ which are equal to $+1$
    (resp.~$-1$).
  \item Let $F$ be a number field and $\rho: \Gamma_F \to \gln(K)$ a
    continuous representation, where $K$ is any topological field of
    characteristic $\neq 2$. We say that $\rho$ is
    \emph{$\glo$} \footnote{This is often simply called \emph{odd} in
      the literature, but we have added $\gln$ here to distinguish
      this notion from the stronger notion of oddness used in
      \cite{fkp-duke} for representations into general reductive
      groups $G$.}  if for every real place $v$ of $F$ with
    $c_v \in \Gamma_{F_v}$ the corresponding complex conjugation, the
    involution $\rho(c_v)$ is odd as in (1).
  \end{enumerate}
  
\end{defn}
All Galois representations corresponding to cohomological automorphic
forms or the cohomology (possibly with torsion coefficients) of
locally symmetric spaces for $\gln/F$ are $\glo$ when $F$ is totally
real by results of Caraiani and Le Hung \cite{caraiani-lehung}, so
there is a plethora of such representations.

Trianguline representations were introduced by Colmez
\cite{colmez-tri}, motivated by work of Kisin \cite{kisin-over} on
Galois representations attached to overconvergent modular forms, and
they are closely related to the theory of eigenvarieties. For example,
Hansen has conjectured \cite[Conjecture 1.2.3]{hansen-eig} that any
(semisimple) $\glo$ representation $\rho: \Gamma_F \to \gln(E)$ which
is unramified outside a finite set of primes and is trianguline at all
primes above $p$ arises as the representation associated (in
\cite[Theorem B]{joh-newt-ext}) to a point on one of the
eigenvarieties for $\gln/F$ constructed in \cite{hansen-eig}, so the
condition on the lifts that we impose is a natural
weakening\footnote{Almost: all crystalline, and even semi-stable,
  representations are trianguline, but not all de Rham representations
  are trianguline, see \cite{berger-triang}.} of geometricity.

For $F$ any number field and $\mc{S}$ a finite set of finite places of
$F$ we let $\Gamma_{F, \mc{S}}$ denote the Galois group of the maximal
extension of $F$ (inside a fixed algebraic closure) unramified outside
all places in $\mc{S}$ (and the infinite places).  The following
is a special case of the main result of this note.
\begin{thm} \label{t:main} %
  Let $F$ be an arbitrary number field and
  $\br:\Gamma_{F,\mc{S}} \to \gln(k)$ a $\glo$ representation. If
  $p \gg_{n,F} 0$ and $\br|_{\Gamma_{F(\zeta_p)}}$ is absolutely
  irreducible, then there exists a finite set of places
  $\mc{S}' \supset \mc{S}$ and a finite extension $E'$ of $E$ with
  ring of integers $\co'$ such that $\br$ lifts to a $\glo$
  representation $\rho: \Gamma_{F,\mc{S}'} \to \gln(\co')$ which is
  regular trianguline at all primes of $F$ above $p$. Furthermore, one
  can also ensure that $\rho(\Gamma_F)$ contains an open subgroup of
  $\mr{SL}_n(\co')$.
\end{thm}
In the main text this is Corollary \ref{c:triang}. It is deduced
from Theorem \ref{t:triang} which allows one to construct trianguline
lifts under weaker hypotheses, e.g., for $\br$ which might not be
irreducible.

\smallskip

We recall that for odd irreducible representations over totally real
fields the Khare--Wintenberger method allows one to construct
geometric lifts (for many $G$) without the need for enlarging
$\mc{S}$. However, this depends on potential modularity results but
since (global) trianguline representations do not in general
correspond to classical automorphic forms it is not possible to apply
this to construct trianguline lifts.

\smallskip

Recently, a definition of trianguline representations with values in
general (connected) reductive groups $G$ has been given by Vincent de
Daruvar in his thesis \cite{daruvar}; one expects that these are
related to eigenvarieties for groups that are not forms of
$\gln$. Using his results we can extend the above theorem to general
(connected) reductive groups, see Theorem \ref{t:triang2}, with the
caveat that the main result of \emph{loc. cit.} depends on two
assumptions: the first is the existence of a sufficiently general
trianguline lift of $\br$ and the second is that $G$ should have no
factors of type $\mr{G}_2$, $\mr{F}_4$ and $\mr{E}_8$. One expects
that the first assumption always holds and the second is unnecessary.

It is an interesting question to extend the definition of trianguline
representations to disconnected groups $G$ and to prove analogues of
the results of \cite{daruvar} in this case. If this is done, the
methods of \cite{fkp-duke} and \cite{fkp-invent} should allow
one to construct global trianguline lifts also for such $G$.

We also mention that beginning with the article \cite{aip-halo}, there
has been interest in trianguline representations over $k[[t]]$, where
$k$ is a finite field; see in particular \cite[\S 1.2.3]{aip-halo}.
It is then a natural question as to whether a version of Theorem
\ref{t:main} holds with $\mc{O}$ above replaced by $k[[t]]$. We intend
to pursue this in the near future.

\subsection{}

Our proofs are based on the methods developed in \cite{fkp-duke} and
\cite{fkp-invent}, and as far as the global arguments go there
are no essential changes, except that in Section \ref{s:3} we
axiomatise in \S \ref{s:selmer} and \S\ref{s:adequate} the conditions
under which the methods of \emph{loc.~cit.} lead to the construction
of lifts, unramified outside a finite set of primes, of a very general
class of Galois representations. The main improvements are in the
local arguments, so we describe these briefly here.

In the original lifting arguments of Ramakrishna \cite{ramakrishna02}
(which are generalized in \cite{fkp-duke} and \cite{fkp-invent}) it is
assumed that one is given smooth local conditions at all primes at
which $\br$ is ramified, i.e., smooth quotients of the framed local
deformation rings, and these should have sufficiently large
dimension. One of the key improvements of this method in
\cite{fkp-duke} was that we were able to dispense with the smoothness
condition, but the dimension condition cannot be relaxed if we want
our lifts to be geometric. However, if we only require that our local
lifts be trianguline at primes dividing $p$ then the dimension
condition can indeed be relaxed, but we face the problem that
trianguline lifts are not parametrised by a quotient of the universal
framed deformation ring. \S \ref{s:cocy} is devoted to showing how we
can avoid this difficulty: in \S \ref{s:cocyc} we explain how the
method of \cite[\S 4]{fkp-duke} can be modified to deal with more
general local conditions, and in \S \ref{s:bhs} we show that this can
be applied in the setting of trianguline lifts using the properties of
the ``trianguline variety'' of Breuil--Hellmann--Schraen
\cite[Th\'eor\`eme 2.6]{bhs-int} (and the generalisation thereof in
\cite{daruvar}). Once these local improvements are in place the proofs
of the main results follow by specialising Theorem \ref{t:general} to
the case where the local lifts are assumed to be trianguline.

\smallskip

\emph{Acknowledgements.} We thank Sandeep Varma for making us aware of
\cite{daruvar} and Vincent de Daruvar for making it available to us.

\section{Cocycle constructions} \label{s:cocy}

\subsection{A general cocycle construction}  \label{s:cocyc}

Let $E$ be a finite extension of $\Q_p$, $\co$ its ring of
integers, $m_{\co} = (\vpi)$ its maximal ideal and
$k = \co/m_{\co}$ its residue field. Let $G$ be a reductive
group scheme over $\co$ (possibly disconnected), and let $\mg{g}$ be
its Lie algebra. Let $G^{\der}$ be the derived group of the identity
component of $G$ and $\fgder$ its Lie algebra.

Let $\Gamma$ be any topologically finitely generated profinite group.
Given a continuous homomorphism $\br: \Gamma \to G(k)$, there exists
a complete local Noetherian $\co$-algebra $R^{\sq,\un}$ with residue
field $k$ and a homomorphism $\rho^{\un}: \Gamma \to G(R^{\sq,\un})$
representing the functor of lifts of $\br$ on the category of complete
local Noetherian $\co$-algebras with residue field $k$. For any
such algebra $A$, we set $\wh{G}(A) := \ker(G(A) \to G(k))$. The
conjugation action of $\wh{G}(\co)$ on $G(R^{\sq,\un})$ and the
universal property of $\rho^{\un}$ induces an action of
$\wh{G}(\co)$ on $R^{\sq,\un}$.

Let $R$ be a quotient of $R^{\sq,\un}$ by an ideal which is invariant
under this action. We assume that $R$ is reduced and flat over
$\co$. We also assume that $R$ corresponds to ``fixed-multiplier''
liftings of $\br$, by which we mean that the map
$\Gamma \to G/G^{\der}(R)$ induced from $\rho^{\sq,\un}$ by the maps
$R^{\sq,\un} \to R$ and $G \to G/G^{\der}$, factors through a fixed
map $\Gamma \to G/G^{\der}(\co)$.\footnote{This is not essential, but
  it is convenient for our applications to impose this condition. The
  results below hold in general if one removes all occurrences of the
  superscript ``$\der$''.}

For any representation $\rho: \Gamma \to G(\co)$ we will denote by
$\rho_r$ the reduction of $\rho$ modulo $\vpi^r$. Also, for any
representation $\rho_r: \Gamma \to G(\co/\vpi^r)$, $\rho_r(\fgder)$
will denote $\fgder \otimes_{\co} \co/\vpi^r$ equipped with the
$\Ad \circ \rho_r$ action.  \smallskip

The following result is essentially contained in \cite[\S
4]{fkp-duke}.
\begin{prop} \label{prop1}
  Assume that $\spec(R)$ has an $\co$-valued point $y$ such that
  the corresponding point of $\spec(R[1/\vpi])$ is contained in the
  smooth locus, and let $\rho: \Gamma \to G(\co)$ be the
  corresponding lift of $\br$. Then there is an open (in the $p$-adic
  topology) subset $Y \subset \spec(R)(\co)$ with $y \in Y$ having
  the following properties: let $Y_n$ be the image of $Y$ in
  $\spec(R)(\co/\vpi^n)$ and for integers $n, r \geq 0$ let
  $\pi^Y_{n,r}: Y_{n+r} \to Y_n$ be the natural maps.
  \begin{enumerate}
  \item Given $r_0 > 0$ there exists $n_0 >0$ such that for all
    $n \geq n_0$ and $0 \leq r \leq r_0$ the fibers of $\pi_{n,r}^Y$
    are nonempty principal homogenous spaces over a submodule
    $Z_r \subset Z^1(\Gamma, \rho_r(\fgder))$ which is free over $\co/\vpi^r$ of rank
    $d$, where $d$ is the dimension of $\spec(R[1/\vpi])$ at $y$.
  \item $B^1(\Gamma, \rho_r(\fgder))
    \subset Z_r$.
  \item The $\co$-module inclusions
    $\co/\vpi^{r-1} \to \co/\vpi^r$, mapping $1$ to $\vpi$,
    and the surjections $\co/\vpi^{r} \to \co/\vpi^{r-1}$
    induce inclusions $Z_{r-1} \to Z_r$ and surjections
    $Z_r \to Z_{r-1}$.
  \item Let $L_r$ be the image of $Z_r$ in
    $H^1(\Gamma, \rho(\fgder) \otimes \co/\vpi^r)$. The groups
    $L_r$ are compatible with the maps on cohomology induced by the
    inclusions $\co/\vpi^{r-1} \to \co/\vpi^r$ and the
    surjections $\co/\vpi^r \to \co/\vpi^{r-1}$.
  \item
    $|L_r| = |\co/\vpi^r|^d\cdot |\rho_r(\fgder)^{\Gamma}|
    \cdot|\rho_r(\fgder)|^{-1} $.
\end{enumerate}
\end{prop}

\begin{proof}
  This is proved in \cite[Proposition 4.7]{fkp-duke} (see
  \cite[Lemma 4.5]{fkp-duke} for (2)) for $\Gamma$ the absolute
  Galois group of a local field and some specific choices of rings $R$
  for which the dimension $d$ is also known, but the proof given there
  works without any changes under the conditions that we have
  given. In particular, the proof of the formula for $|L_r|$ there
  extends to give the formula in (5).
\end{proof}

The goal of this subsection is to formulate a generalisation of
Proposition \ref{prop1} which will be the key to our construction of
trianguline lifts in \S\ref{s:lifts}, allowing us to circumvent the fact that there is no quotient of $R^{\square, \un}$ parametrizing the trianguline lifts of $\br$. In the next Proposition, $R$ is as above, but we emphasize that the manifold $U$ will no longer be assumed open in $\spec(R[1/\vpi])$. For $m > 0$, let
$(\wh{G^{\der}})^{(m)}(\co) = \ker(G^{\der}(\co) \to
G^{\der}(\co/\vpi^m))$. We then have the following: 

\begin{prop} \label{prop2}
  Let $U$ be a compact $E$-adic manifold of some dimension $d$ contained in
  the smooth locus of $\spec(R[1/\vpi])(E)$ with $v \in U$, and let
  $\rho$ be the lift of $\br$ corresponding to $v$.  Assume that there
  exists $m \gg 0$ such that the $(\wh{G^{\der}})^{(m)}(\co)$-orbit of
  $v$ is contained in $U$.  Then there is an open (in the $p$-adic
  topology) subset $V \subset U$ with $v \in V$ having the following
  properties: let $V_n$ be the image of $V$ in $\spec(R)(\co/\vpi^n)$
  and for integers $n, r \geq 0$ let $\pi^V_{n,r}: V_{n+r} \to V_n$ be
  the natural maps.
  \begin{enumerate}
  \item Given $r_0 > 0$ there exists $n_0 >0$ such that for all
    $n \geq n_0$ and $0 \leq r \leq r_0$ the fibers of $\pi_{n,r}^V$
    are nonempty principal homogenous spaces over a submodule
    $Z_r \subset Z^1(\Gamma, \rho(\fgder) \otimes_{\co}
    \co/\vpi^r)$ which is free of rank $d$ over
    $\co/\vpi^r$.
  \item $B^1(\Gamma, \rho(\fgder) \otimes_{\co} \co/\vpi^n)
    \subset Z_r$.
  \item The $\co$-module inclusions
    $\co/\vpi^{r-1} \to \co/\vpi^r$, mapping $1$ to $\vpi$,
    and the surjections $\co/\vpi^{r} \to \co/\vpi^{r-1}$
    induce inclusions $Z_{r-1} \to Z_r$ and surjections
    $Z_r \to Z_{r-1}$.
  \item Let $L_r$ be the image of $Z_r$ in
    $H^1(\Gamma, \rho(\fgder) \otimes \co/\vpi^r)$. The groups
    $L_r$ are compatible with the maps on cohomology induced by the
    inclusions $\co/\vpi^{r-1} \to \co/\vpi^r$ and the
    surjections $\co/\vpi^r \to \co/\vpi^{r-1}$.
  \item
     $|L_r| = |\co/\vpi^r|^d\cdot |\rho_r(\fgder)^{\Gamma}|
    \cdot |\rho_r(\fgder)|^{-1} $.
\end{enumerate}
\end{prop}

\begin{proof}
  Fix $r_0$. We then apply Proposition \ref{prop1} by taking $y=u$
  there and get an open set $Y \subset \spec(R[1/\vpi])(\co)$ with
  $u \in Y$, an integer $n_0^Y$ and cocycles $Z_r^Y$ satisfying all
  the conclusions there.

  Replacing $U$ by $U \cap Y$ we may assume that $U \subset Y$. By
  applying the result of Serre \cite[Proposition 11]{serre:cebotarev}
  to the inclusion $v \in U$ in the same way as in the proof of
  \cite[Lemma 4.3]{fkp-duke} we get an open submanifold
  $V \subset U$ with $v \in V$ and free $\co/\vpi^r$ submodules
  $Z_r \subset Z_r^Y$ with the property in (1), with the proviso that
  the $n_0$ associated to $r_0$ might be greater than the $n_0^Y$
  obtained as an output of Proposition \ref{prop1}. (The inclusion
  $Z_r \subset Z_r^Y$ follows from the definitions since $U \subset Y$
  so each $U_n \subset Y_n$.)

  Using the assumption on the  $(\wh{G^{\der}})^{(m)}(\co)$-orbit
  of $v$ we see that (2) holds in the same way as in the proof of
  \cite[Lemma 4.5]{fkp-duke}.

  The surjectivity part of (3) is clear from the defining property of
  the $Z_r$'s. To prove the injectivity part, we must examine the
  proof of \cite[Lemma 4.3]{fkp-duke} (since there is no analogue of
  \cite[Lemma 4.4]{fkp-duke} in the present setting) from which we
  obtain the $Z_r$. That proof proceeds by first choosing a formally
  smooth complete Noetherian $\co$-algebra $A$ and a surjection of
  $\co$-algebras $A \to R$. The $Z_r$ (corresponding to $U$) are
  then constructed as certain submodules of
  $\hom_{\co}(\Omega_{A/\co} \otimes_{A,y} \co, \co/
  \vpi^r)$ by choosing local equations for $U$ in the $E$-adic
  manifold $\spec(A)(\co)$ and reducing to the case that $U$
  corresponds to the $E$-valued points of a formally smooth quotient
  $R'$ of $A$. It is an immediate consequence of this construction
  that the map
  \[
     \hom_{\co}(\Omega_{A/\co} \otimes_{A,y} \co, \co/
  \vpi^{r-1}) \to \hom_{\co}(\Omega_{A/\co} \otimes_{A,y} \co, \co/
  \vpi^{r})
\]
induced by the inclusion $\co/\vpi^{r-1} \to \co/\vpi^r$
maps $Z_{r-1}$ injectively into $Z_r$ (since this is true when $A$ is
replaced by $R'$). The map $A \to R$ induces compatible (with respect
to $\co/\vpi^{r-1} \to \co/\vpi^r$) inclusions
$\hom_{\co}(\Omega_{R/\co} \otimes_{R,y} \co, \co/
\vpi^{r}) \to \hom_{\co}(\Omega_{A/\co} \otimes_{A,y} \co,
\co/ \vpi^{r})$ from which (3) follows once we identify
$\hom_{\co}(\Omega_{R/\co} \otimes_{A,y} \co, \co/
\vpi^{r})$ with a submodule of
$Z^1(\Gamma, \rho(\fgder) \otimes \co/\vpi^r)$ as in \cite[\S
4.2]{fkp-duke}.

Finally, (4) follows immediately from (3) and the definition of
$L_r$ and (5) follows as in Proposition \ref{prop1}.
\end{proof}

\subsection{Cocyles for trianguline deformations} \label{s:bhs}

Let $F$ be a finite extension of $\Q_p$ and let
$\Gamma_F = \gal(\ov{F}/F)$ be its absolute Galois group.  In this
subsection we recall some facts about trianguline representations of
$\Gamma_F$ and carry out the local analysis needed in order to be able
to construct the cocyles used to prove the existence of global lifts
of $\glo$ representations which are trianguline at primes above
$p$. For a general introduction to trianguline representations the
reader may consult \cite{berger-triang} (and also \cite{daruvar} for
the $G$-valued case), but we will only use the existence and some
properties of the rigid-analytic variety of (local) trianguline lifts
from \cite{bhs-int} (and \cite{daruvar} for the $G$-valued
case). However, to orient the reader not familiar with this notion, we
mention briefly that trianguline representations are defined by
embedding the category of $E$-adic representations of $\Gamma_F$ (for
$E/\Q_p$ a finite extension) fully faithfully into a larger category,
the category of $(\phi, \Gamma)$-modules, which are modules over a
(Robba) ring $\mc{R}$ with some extra structure. A Galois
representation is said to be trianguline if the corresponding
$(\phi, \Gamma)$-module has a filtration by subobjects such that the
successive subquotients are free of rank one as $\mc{R}$-modules; the
parameter $\delta$ occurring below corresponds to the ordered tuple of
these subquotients. The condition for a $G$-valued representation to
be trianguline is formulated in Tannakian terms using the tensor
structure on the category of $(\phi, \Gamma)$-modules or in the
language of principal $G$-bundles over $\mc{R}$ (with extra
structure).

\smallskip

\subsubsection{}
We will first consider the case of $\gln$ since the results for
general $G$ are slightly weaker and partly conditional.  So let
$G = \gln$ and let $\br: \Gamma_F \to G(k)$ be as in \S
\ref{s:cocyc}. Let $\mg{X}^{\sq}_{\br}$ be the rigid analytic space
over $E$ corresponding to the formal scheme $\spf(R^{\sq,\un})$, where
$R^{\sq,\un}$ is the universal lifting ring of $\br$; its $E$-valued
points are canonically equal to the $\co$-valued points of
$\spec(R^{\sq,\un})$ and similarly for all finite exensions of
$E$. Let $\mc{T}$ denote the rigid analytic space over $\Q_p$
parametrising the continuous characters of $F^{\times}$. Let
$X^{\sq}_{\tri}(\br) \subset \mg{X}^{\sq}_{\br} \times \mc{T}_E^n$ be
the space of trianguline deformations of $\br$ as in
\cite[D\'efinition 2.4]{bhs-int} (but our notation is slightly
different). It is a Zariski closed rigid analytic subvariety of
$\mg{X}^{\sq}_{\br} \times \mc{T}_E^n$ defined as the Zariski closure
of a subset $U^{\sq}_{\tri}(\br)^{\reg}$. The points $(x, \delta)$ of
the latter consist of certain trianguline lifts $x$ of $\br$ together
with a system of parameters $\delta$ of a triangulation of the
associated $(\phi, \Gamma)$-module; the assumption is that $\delta$ is
regular in the sense explained in the paragraph before
\cite[D\'efinition 2.4]{bhs-int}. For our purposes what is important
is that $x$ corresponds to a trianguline lift, and the precise nature
of $\delta$ plays no explicit role.  We call the trianguline lifts of
$\br$ which correspond to points of $X^{\square}_{\mr{tri}}(\br)$
\emph{good}; we do not know whether all trianguline lifts are good.

The following result of Breuil, Hellmann and Schraen is the key input
for the construction of cocyles for trianguline lifts.
\begin{lem} \label{l:bhs}
\hspace{-1cm}\begin{enumerate}
\item The space $X^{\sq}_{\tri}(\br)$ is non-empty and equidimensional
  of dimension $n^2 + [F:\Q_p]\tfrac{n(n+1)}{2}$.
\item  $U^{\sq}_{\tri}(\br)^{\reg}$ is a smooth Zariski open subvariety of  $X^{\sq}_{\tri}(\br)$
  which is Zariski dense.
\item The projection map
  $\pi_1:X^{\sq}_{\tri}(\br) \to \mg{X}^{\sq}_{\br}$ is an immersion
  at all points of $U^{\sq}_{\tri}(\br)^{\reg}$.
\end{enumerate}
\end{lem}

\begin{proof}
  This is a part of \cite[Th\'eor\`eme 2.6]{bhs-int}. The third part
  is not contained in the statement there but is used in the proof,
  where it is deduced from results of Bella\"iche and Chenevier
  \cite[\S 2.3]{bc-selmer}.  We note that the proof of non-emptiness
  of $U^{\sq}_{\tri}(\br)^{\reg}$ in \cite{bhs-int} uses the existence
  of regular crystalline lifts of $\br$; this is a highly non-trivial
  result for general $\br$ and was proved only recently in
  \cite{emerton-gee:moduli}.
\end{proof}

The group $\wh{\gln}(\co)$ acts on
$\mg{X}^{\sq}_{\br} \times \mc{T}_E^n$ via its action on the first
factor and this action preserves $U^{\sq}_{\tri}(\br)^{\reg}$ and $X^{\sq}_{\tri}(\br)$.

For our application we need to work with lifts of $\br$ which have a
fixed determinant. To arrange this we consider the morphism
$\det: \mg{X}^{\sq}_{\br} \to \mg{X}^{\sq}_{\det(\br)}$ which sends
any lift of $\br$ to its determinant and the induced morphism
$\det_1:= \det \circ \pi_1:X^{\sq}_{\tri}(\br) \to
\mg{X}^{\sq}_{\det(\br)}$. The space $\mg{X}^{\sq}_{\det(\br)}$ has
dimension $1 + [F:\Q_p]$ and the morphism $\det_1$ commutes with the
action of $\wh{\gln}(\co)$, where the action on
$\mg{X}^{\sq}_{\det(\br)}$ is taken to be the trivial action.

If $\chi$ is a character of $\Gamma_F$ with values in a finite
extension $E'$ of $E$ with trivial reduction modulo its maximal ideal,
then tensor product with $\chi$ induces an automorphism of
$X^{\sq}_{\tri}(\br)_{E'}$ which preserves
$U^{\sq}_{\tri}(\br)^{\reg}$. This is compatible with the morphism
$\det_1$ if we let $\chi$ act on $\mg{X}^{\sq}_{\det(\br)}$ by tensor
product with $\chi^{\otimes n}$.

If $p \nmid n$, then any character $\chi$ as above has an $n$-th
root. This implies that in this case all the fibres of $\det_1$ over
$\mg{X}^{\sq}_{\det(\br)}(E')$ are isomorphic, for any finite
extension $E'$ of $E$. It follows that for any point $y$ of
$\mg{X}^{\sq}_{\det(\br)}$,
$U^{\sq}_{\tri}(\br)^{\reg} \cap (\det_1)^{-1}(y)$ is Zariski dense in
$(\det_1)^{-1}(y)$.

\begin{lem} \label{l:triang} Let $\br: \Gamma_F \to \gln(k)$ be a
  continuous homomorphism, and let $\rho:\Gamma_F \to \gln(\co)$ be a
  lift of $\br$ corresponding to a point
  $(x,\delta) \in X^{\sq}_{\tri}(\br)(E)$. Assume $p \nmid
  n$. Then, after replacing $E$ by a finite extension if necessary,
  the following holds:
  \begin{enumerate}
  \item There is a quotient $R$ of $R^{\sq,\un}$ by an ideal invariant
    under the action of $\wh{\gln}(\co)$ which is reduced and
    flat over $\co$.
  \item There is a compact $E$-adic submanifold $U$ of
    $\spec(R[1/\vpi])(E)$ of dimension
    $d = n^2 -1 + [F:\Q_p]\bigl (\tfrac{n(n+1)}{2} - 1 \bigr )$ such
    that all the points of $U$ correspond to regular trianguline lifts
    $\rho'$ of $\br$ with $\det(\rho) = \det(\rho')$.
  \item Given any integer $N> 0$, $U$ can be chosen such that all
    $\rho'$ corresponding to points of $U$ are congruent to $\rho$
    modulo $\vpi^N$.
  \item There is a point $v \in U$ and an integer $m > 0$ such that
    the $(\wh{\gln})^{(m)}(\co)$-orbit of $v$ is contained in
    $U$.
  \end{enumerate}
\end{lem}

\begin{proof}
  By replacing $E$ with a finite extension if necessary, we may (and
  do) ensure that $U^{\sq}_{\tri}(\br)^{\reg}(E) \neq \emptyset$, so
  by Lemma \ref{l:bhs} it has a natural structure of an $E$-adic
  manifold of dimension $n^2 + [F:\Q_p]\tfrac{n(n+1)}{2}$. Since the
  space $\mg{X}^{\sq}_{\det(\br)}$ has dimension $1 + [F:\Q_p]$ and
  all (non-empty) fibres of $\det_1$ are geometrically isomorphic
  (since $p \nmid n$), it follows from the results of \cite[\S
  III.10]{serre-lalg} that
  $\mc{F} := (\det_1)^{-1}(\det_1((x,\delta)))$ is of pure dimension
  $d = n^2 -1 + [F:\Q_p]\bigl (\tfrac{n(n+1)}{2} - 1 \bigr )$ and has
  a smooth Zariski dense open subset $\mc{F}'$ consisting of points in
  $U^{\sq}_{\tri}(\br)^{\reg}$.

  Now using \cite[Lemma 4.9]{fkp-duke}\footnote{The lemma as stated
    does not quite apply here, but its proof gives what we
    claim.}---this might again require replacing $E$ with a finite
  extension---we can find a sequence of points
  $(x_N', \delta_N') \in \mc{F}'(E)$, $N=1,2,\dots$, such that the
  lift of $\br$ corresponding to $x_N'$ is congruent to $\rho$ modulo
  $\vpi^N$.  For any such $N$, let $U' \subset \mc{F}'(E)$ be an $E$-adic neighbourhood
  of $(x'_N, \delta'_N)$ such that all the lifts of $\br$
  corresponding to points of $U'$ are congruent to $\rho$ modulo
  $\vpi^N$. By shrinking $U'$ further using (3) of Lemma \ref{l:bhs},
  we may and do assume that $\pi_1$ induces an embedding of $U'$ onto
  a compact $E$-adic submanifold $U$ of
  $\spec(R^{\sq,\un})[1/\vpi])(E)$.

  Let $\wt{U'} = \wh{\gln}(\co)\cdot U' \subset \mc{F}'(E)$ and let
  $R$ be the reduced quotient of $R^{\sq,\un}$ corresponding to the
  Zariski closure in $\spec(R^{\sq,\un})$ of $\pi_1(\wt{U'})$.  Since
  $\pi_1(\wt{U'})$ is preserved by the action of $\wh{\gln}(\co)$, the
  kernel of the quotient map $R^{\sq,\un} \to R$ is also preserved by
  this group. By the Zariski density of $\pi_1(\wt{U'})$ in $\spec(R)$
  and the reducedness of $\spec(R)$, we may find a point
  $v \in \pi_1(\wt{U'})$ which is a smooth point of
  $\spec(R[1/\vpi])$. The $\wh{\gln}(\co)$-equivariance of $\pi_1$
  implies that we may take $v = \pi_1(v')$ for some $v' \in U'$.
  Replacing $U'$ by a compact open neighbourhood of $v'$ we may
  assume that $\pi_1(U')$ is contained in the smooth locus of
  $\spec(R[1/p])(E)$.

  Letting $U = \pi_1(U')$ and $v = \pi_1(v')$, it is clear from the
  above that (1), (2) and (3) hold.  To prove (4), we note that
  $\wh{\gln}(\co)$ acts continuously on $\mc{F}'(E)$. Since $U'$ is an
  open neighbourhood of $v'$ in $\mc{F}'(E)$, and the sequence of
  subgroups $(\wh{\gln})^{(m)}(\co)$ of $\wh{\gln}(\co)$ forms a
  neighbourhood basis of the identity, for all $m \gg 0$ the
  $(\wh{\gln})^{(m)}(\co)$-orbit of $v'$ is contained in $U'$. The
  $\wh{\gln}(\co)$-equivariance of $\pi_1$ then implies that for
  all $m \gg 0$ the $(\wh{\gln})^{(m)}(\co)$-orbit of $v$ is contained
  in $U$.

\end{proof}

\begin{rem} \label{r:ord}%
  If $\br$ is upper triangular with respect to some basis, then the
  points of $U^{\sq}_{\tri}(\br)^{\reg}$ corresponding to pairs
  $(x,\delta)$ with $x$ being upper triangular with respect to some
  basis form a non-empty $\wh{\gln}(\co)$-invariant open subspace:
  this is a consequence of the inductive nature of the proof of
  existence of regular crystalline lifts in
  \cite{emerton-gee:moduli}. In this case one may arrange that all the
  points of the $U$ constructed in Lemma \ref{l:triang} correspond to
  lifts of $\br$ which are upper triangular with respect to some
  basis.
\end{rem}

\subsubsection{}
Now suppose $G$ is a connected split reductive group over a $p$-adic
field $E$. Then the notion of trianguline representations of
$\Gamma_F$ valued in $G(E)$ is defined in \cite{daruvar} using $(\phi,
\Gamma)$-modules with $G$-structure, which are defined using
Tannakian methods. We do not recall the details of the construction,
but only mention that this definition generalizes the usual definition
of trianguline representations in a natural way. The main result of
\cite{daruvar}, Theorem 6.22, is an analogue for $G$-valued
representations of \cite[Th\'eor\`eme 2.6]{bhs-int}, with some
modifications and conditions which we briefly explain.

For $\br: \Gamma_F \to G(k)$, we let $R^{\sq,\un}$ and
$\mg{X}^{\sq}_{\br}$ be as before. For $T$ a maximal split torus of
$G$ contained in a Borel subgroup $B$, we let $T^{\vee}$ be the dual
torus and we replace $\mc{T}_E^n$ by $\wh{T^{\vee}}$, the rigid
analytic space over $E$ parametrising (continuous) characters of
$T^{\vee}(F)$. Then
$X^{\sq}_{\tri}(\br) \subset \mg{X}^{\sq}_{\br} \times \wh{T^{\vee}}$
is defined as the Zariski closure of a subset
$U^{\sq}_{\tri}(\br)^{\mr{vreg}}$ of \emph{very regular} lifts. The
points of the latter correspond to pairs $(x, \delta)$ where $x$ is a
trianguline lift of $\br$ in the sense of \cite[Definition
4.9]{daruvar} with a triangulation such that the associated parameter
$\delta$ is very regular in the sense of \cite[Definition
6.1]{daruvar}. We then have the following analogue of Lemma
\ref{l:bhs}.

\begin{lem} \label{l:daruvar}
  Assume $G$ has no factors of type $\mr{G}_2$, $\mr{F}_4$ and
  $\mr{E}_8$ and $\br$ has a very regular trianguline lift. Then
\hspace{-1cm}\begin{enumerate}
\item The space $X^{\sq}_{\tri}(\br)$ is non-empty and equidimensional
  of dimension $\dim(G) + [F:\Q_p]\dim(B)$.
\item  $U^{\sq}_{\tri}(\br)^{\mr{vreg}}$ is a smooth Zariski open subvariety of  $X^{\sq}_{\tri}(\br)$
  which is Zariski dense.
\item The projection map
  $\pi_1:X^{\sq}_{\tri}(\br) \to \mg{X}^{\sq}_{\br}$ is an immersion
  at all points of $U^{\sq}_{\tri}(\br)^{\mr{vreg}}$.
\end{enumerate}
\end{lem}
It is expected that the condition on $G$ is unnnecessary, the
condition on $\br$ is always satisfied, and ``very regular'' can be
replaced by ``regular''.

\begin{proof}
  This is part of \cite[Theorem 6.22]{daruvar}, the last part being a
  consequence of the proof using \cite[Proposition 6.6]{daruvar}.
\end{proof}

As in the case of $\gln$, we would like to work with ``fixed
determinant'' lifts, where now by ``determinant'' we mean the map
$\det: G \to G/G^{\mr{\der}}$, with $G^{\mr{\der}}$ being the derived
group of $G$. This induces a map
$\det: \mg{X}^{\sq}_{\br} \to \mg{X}^{\sq}_{\det(\br)}$ which sends
any lift of $\br$ to its determinant, and we have an induced morphism
$\det_1:= \det \circ \pi_1:X^{\sq}_{\tri}(\br) \to
\mg{X}^{\sq}_{\det(\br)}$. Let $n$ be the order of the kernel of the
isogeny $Z(G) \to G/G^{\mr{\der}}$, where $Z(G)$ is the connected
centre of $G$. 
\begin{lem} \label{l:triang2} Assume $G$ has no factors of type
  $\mr{G}_2$, $\mr{F}_4$ and $\mr{E}_8$, and that $p \nmid n$. 
  Let $\br: \Gamma_F \to G(k)$ be a continuous homomorphism and 
  let $\rho:\Gamma_F \to G(\co)$ be a
  lift of $\br$ corresponding to a point
  $(x,\delta) \in X^{\sq}_{\tri}(\br)(E)$. Then, after replacing $E$
  by a finite extension if necessary, the following holds:
  \begin{enumerate}
  \item There is a quotient $R$ of $R^{\sq,\un}$ by an ideal invariant
    under the action of $\wh{G}(\co)$ which is reduced and
    flat over $\co$.
  \item There is a compact $E$-adic submanifold $U$ of
    $\spec(R[1/\vpi])(E)$ of dimension
    $d = \dim(G^{\mr{der}})  + [F:\Q_p]\bigl (\dim(B)  - \dim(Z(G)) \bigr )$ such
    that all the points of $U$ correspond to very regular trianguline
    lifts $\rho'$ of $\br$ with $\det(\rho) = \det(\rho')$.
  \item Given any integer $N> 0$, $U$ can be chosen such that all
    $\rho'$ corresponding to points of $U$ are congruent to $\rho$
    modulo $\vpi^N$.
  \item There is a point $v \in U$ and an integer $m > 0$ such that
    the $\wh{G}^{(m)}(\co)$-orbit of $v$ is contained in
    $U$.
  \end{enumerate}
\end{lem}

\begin{proof}
  This is proved in essentially the same way as Lemma \ref{l:triang}
  using Lemma \ref{l:daruvar} instead of Lemma \ref{l:bhs}.
\end{proof}

\section{Trianguline lifting theorems} \label{s:3}

In this section we formulate a general lifting theorem for
representations of global fields.

\subsection{Selmer and relative Selmer groups} \label{s:selmer}

Let $F$ be a global field, and let $E, \co, \vpi$ be as before. Let
$G$ be a reductive group scheme over $\co$, $G^{\mr{der}}$ the derived
group of its identity component and $\fgder$ the Lie algebra of
$G^{\mr{der}}$. For some $n \geq 1$ let
$\rho_n: \Gamma_F \to G(\co/\vpi^n)$ be a continuous homomorphism.
Let $\mc{S}$ be a finite set of primes containing all primes at which
$\rho_n$ is ramified and also all primes dividing $p$. We first recall
some definitions and results from \cite{fkp-duke}.

In what follows, when we have an integer $0 < r < n$, we will write
$\rho_{r}$ for the reduction $\rho_n\pmod{\vpi^{r}}$. For each prime
$v \in {\mc{S}}$ we assume that for $0 < r \leq n$ we have subgroups
$Z_{r,v} \subset Z^1(\Gamma_{F_v}, \rho_r(\fgder))$ such that
\begin{itemize}
\item Each $Z_{r, v}$ contains the group of boundaries $B^1(\Gamma_{F_v}, \rho_r(\fgder))$.
\item As $r$ varies, the inclusion and reduction maps induce short exact sequences
\[
0 \to Z_{a, v} \to Z_{a+b, v} \to Z_{b, v} \to 0
\]
as in Proposition \ref{prop2}.
\end{itemize}
We let $L_{r,v} \subset H^1(\Gamma_{F_v}, \rho_r(\fgder))$ be the
image of $Z_{r,v}$, and we let
$L_{r,v}^{\perp} \subset H^1(\Gamma_{F_v}, \rho_r(\fgder)^*)$ be the
annihilator of $L_{r,v}$ under the local duality pairing.  Let
$\mc{L}_r =\{L_{r,v}\}_{v \in {\mc{S}}}$, $0 < r \leq n$, and
similarly define $\mc{L}_r^{\perp}$. The Selmer group
$H^1_{\mc{L}_r}(\Gamma_{F,{\mc{S}}},\rho_r(\fgder))$ is defined to be
\[
\ker \left ( H^1(\Gamma_{F,{\mc{S}}}, \rho_r(\fgder)) \to \bigoplus_{v \in {\mc{S}}}
\frac{H^1(\Gamma_{F_v}, \rho_r(\fgder))}{L_{r,v}} \right )
\]
and the dual Selmer group
$H^1_{\mc{L}_r^{\perp}}(\Gamma_{F,{\mc{S}}}, \rho_r(\fgder)^*)$ is
defined analogously.

\smallskip

The definition below is a variant of the definition of \emph{balanced}
in \cite[Definition 6.2]{fkp-duke} (which is the condition when $a=0$).
\begin{defn} \label{d:sb}
  We say that the local conditions $\mc{L}_r$, for $0 < r \leq n$, are
  \emph{semi-balanced} if there exists a non-negative integer $a$ such
  that
  \[
    |H^1_{\mc{L}_r}(\Gamma_{F,{\mc{S}}}, \rho_r(\fgder))| = |\co/\varpi^r|^a\cdot|
    H^1_{\mc{L}_r^{\perp}}(\Gamma_{F,{\mc{S}}}, \rho_r(\fgder)^*)|
  \]
  for all $0 < r \leq n$.
\end{defn}

The basic objects that we need to control in order to prove lifting
results using the methods of \cite{fkp-duke} and
\cite{fkp-invent} are the relative (dual) Selmer groups of
\cite[Definition 6.2]{fkp-duke}, so we recall them here:
\begin{defn}\label{def:relativeSelmer}
For $0 < r \leq n$, we define the $r$-th relative Selmer
group to be
\[
\ov{H^1_{\mc{L}_r}(\Gamma_{F,{\mc{S}}}, \rho_r(\fgder))} := \im \left ( H^1_{\mc{L}_r}(\Gamma_{F,{\mc{S}}}, \rho_r(\fgder)) \to H^1_{\mc{L}_1}(\Gamma_{F,{\mc{S}}}, \br(\fgder)) \right )
\]
and the $r$-th relative dual Selmer group to be
\[
\ov{H^1_{\mc{L}_r^{\perp}}(\Gamma_{F,{\mc{S}}}, \rho_r(\fgder)^*)} := \im \left ( H^1_{\mc{L}_r^{\perp}}(\Gamma_{F,{\mc{S}}}, \rho_r(\fgder)^*) \to H^1_{\mc{L}_1^{\perp}}(\Gamma_{F,{\mc{S}}}, \br(\fgder)^*) \right ).
\]
\end{defn}

\subsection{Adequate cocyles and semi-balancedness} \label{s:adequate}

We continue with the notation from \S\ref{s:selmer}, but now assume
that $F$ is a number field. We set $\br = \rho_1$ and for each
infinite place $v$ of $F$ we let
$c(\br, v) = \dim(\fgder)^{\Gamma_{F_v}}$.

The definition below is a modification of \cite[Definition
1.4]{fkp-invent}; it is useful for lifting representations
which are more general than the odd representations considered
there.
\begin{defn} \label{d:strong} We say that the $\co$-submodules
  $L_{r,v} \subset H^1(\Gamma_{F_v},\rho_r(\fgder))$ are
  \emph{adequate} if for each finite place $v$ there exists an integer
  $a_v$ such that $a_v = 0$ for all but finitely many $v$,
  \begin{equation}
    |L_{r,v}| = |\rho_r(\fgder)^{\Gamma_{F_v}}|\cdot
    |\co/\vpi^r|^{a_v} ,
  \end{equation}
  and  $\sum_{v \ \mr{finite}} a_v \geq \sum_{v \mid \infty} c(\br, v)$.
  

\end{defn}

The following lemma is a variant of \cite[Lemma 6.3]{fkp-duke}.

\begin{lem} \label{l:semibal} If the collection of
  $\co/\vpi^r$-modules $\mc{L}_r$ is adequate and the spaces of
  invariants $\br(\fgder)^{\Gamma_F}$ and $(\br(\fgder)^*)^{\Gamma_F}$
  are both zero, then the relative Selmer and dual Selmer groups are
  also semi-balanced in the sense that
\[
\dim(\ov{H^1_{\mc{L}_n}(\Gamma_{F,{\mc{S}}}, \rho_n(\fgder))}) \geq \dim(\ov{H^1_{\mc{L}_n^{\perp}}(\Gamma_{F,{\mc{S}}}, \rho_n(\fgder)^*)}).
\]
\end{lem}
\begin{proof}
  The hypotheses together with the Greenberg--Wiles formula
  \cite[Theorem 2.18]{DDT1} imply that the Selmer and dual Selmer
  groups are semi-balanced: the integer $a$ of Definition \ref{d:sb}
  is $\sum_{v \ \mr{finite}}a_v - \sum_{v \mid \infty} c(\br, v)$. From this
  the proof of semi-balancedness of the relative Selmer and dual
  Selmer group follows from \cite[Lemma 6.1]{fkp-duke} as in the
  proof of \cite[Lemma 6.3]{fkp-duke}.
\end{proof}

\subsection{The general lifting theorem} \label{s:general}

As above, let $F$ be a global field and $\br: \Gamma_F \to G(k)$ be a
continuous representation, where $k$ is a finite field of
characteristic $p$ and $G$ is a split reductive group over the ring of
integers $\mc{O}$ of a finite extension $E/\Q_p$; the group $G$ need
not be connected, and if this is the case we assume that the component
group of $G$ has order prime to $p$ and $G$ is a semi-direct product
of its identity component $G^0$ and the component group. Let $\wt{F}$
be the smallest extension of $F$ such that
$\br(\Gamma_{\wt{F}}) \subset G^0(k)$ and let $K = F(\br, \mu_p)$.

We let $\bar{\mu}: \Gamma_F \to (G/G^{\mr{der}})(k)$ be the map induced by
$\br$ and the quotient map $G \to G/G^{\mr{der}}$, and we fix a lift
$\mu: \Gamma_F \to G(\mc{O})$ of $\bar{\mu}$.

\begin{assA} \label{a:1}
  $ $
  \begin{enumerate}
\item $\br$ and $\mu$ are unramified outside a finite set of places
  $\mc{S}$ of $F$ containing all places of $F$ over $p$ if $F$ is a number
  field.
\item If $F$ is a function field we assume that $p \neq \ch(F)$.
\item $H^1(\mr{Gal}(K/F), \br(\fgder)^*) = 0$.
\item $\br(\fgder)$ and $\br(\fgder)^*$ do not contain the trivial
  representation as a submodule.
\item There is no surjection of $\F_p[\Gamma_F]$-modules from
  $\br(\fgder)$ onto any $\F_p[\Gamma_F]$-module subquotient of
  $\br(\fgder)^*$.
\end{enumerate}
\end{assA}

We also need some local assumptions on $\br$ which we formulate
separately.

\begin{assB} \label{a:2} %
  There exists a finite extension $E'$ of $E$ with ring of integers
  $O'$ such that for each finite place $v \in \mc{S}$ there exists a
  lift $\rho_v$ of $\br|_{\Gamma_{F_v}}$ to a continuous homomorphism
  $\Gamma_{F_v} \to G(\mc{O'})$ with fixed determinant
  $\mu|_{\Gamma_{F_v}}$. Furthermore, $\rho_v$ is an element of an
  $E$-adic manifold $U_v$ contained in
  $\spec(R_{\br|_{\Gamma_{F_v}}}^{\sq, \mr{univ}})(\mc{O'})$, with
  fixed determinant as above, and which is invariant under conjugation
  by $\wh{G^{\mr{der}}}(O')$ and of dimension $d_v$, with
  $d_v = \dim_k (\fgder)$ for each finite $v \nmid p$. Finally,
  \begin{equation} \label{e:ad}
    \sum_{v \ \mr{finite}} (d_v - \dim(\fgder)) \geq \sum_{v \mid
      \infty} c(\br, v) .
  \end{equation}
\end{assB}

\begin{rem} \label{r:lifts} $ $
  \begin{enumerate}
  \item We usually choose $U_v$ to lie in the $\mc{O}'$-points of an
    irreducible component of
    $\spec(R)$, where $R$ is a $\wh{G}(\mc{O}')$-invariant (or $\wh{G^{\mr{der}}}(\mc{O}')$-invariant) quotient of the
    universal local lifting ring of $\br|_{\Gamma_{F_v}}$.
  \item The question of the existence of $U_v$ is most delicate for
    primes $v \mid p$, but for arbitrary groups $G$ it is still
    unknown whether a lift $\rho_v$ always exists even for other
    (ramified) primes. In most applications, if $v \nmid p$ then we
    choose $U_v$ so that $d_v = \dim_k(\fgder)$ (see \cite[Proposition
    4.7]{fkp-duke}; this is the maximal possible), in which case
    $a_v = 0$.
  \end{enumerate}
\end{rem}

The following theorem is now an easy consequence of the results of
\cite{fkp-duke} and \cite{fkp-invent}.

\begin{thm} \label{t:general} Let $\br:\Gamma_F \to G(k)$ satisfy
  Assumptions A and B and assume that $p \gg_G
  0$. Then there exists a finite set of places
  $\mc{S}' \supset \mc{S}$ such that $\br$ lifts to a continuous
  representation $\rho: \Gamma_{F, \mc{S}'} \to G(\mc{O}')$
  satisfying:
  \begin{enumerate}
  \item The image $\rho( \Gamma_{F, \mc{S}'})$ intersects
    $G^{\mr{der}}(\mc{O}')$ in an open subgroup.
  \item $\rho|_{\Gamma_{F_v}}$ is an element of $U_v$ for all
    $v \in \mc{S}$ and $\rho$ can be chosen so that $\rho|_{\Gamma_{F_v}}$
    is, modulo $\wh{G^{\mr{der}}}(\mc{O}')$-conjugation, congruent to $\rho_v$ modulo any pre-specified power of $\varpi$.
  \end{enumerate}
\end{thm}

\begin{proof}
  Assumptions A and the existence of $\rho_v$ in Assumption
  B imply that all the arguments of Sections 3 and 4 of
  \cite{fkp-invent} go through without any changes to construct lifts
  $\rho_n$ as in Theorem 4.4 of \emph{loc.~cit.}. The point is that
  neither the archimedean primes, nor the dimensions $d_v$, play any
  role in the proof of this theorem. To complete the proof we must
  explain how to carry out the ``relative deformation theory''
  arguments of \cite[Section 6]{fkp-invent}.

  To do this, we apply Proposition \ref{prop2} to the sets $U_v$ for
  each $v \in \mc{S}$ to produce cocycles $Z_{r,v}$ for all $r \leq n$. The
  condition on $d_v$ in Assumption B and (5) of Proposition
  \ref{prop2} implies that the $L_{r,v}$ constructed from the
  $Z_{r,v}$ are adequate in the sense of Definition
  \ref{d:strong}. Then by Lemma \ref{l:semibal}, the relative Selmer
  and dual Selmer groups are semi-balanced. The first part of
  Assumption 5.1 of \cite{fkp-invent} is part of (4) of Assumption
  A above, and this together with the existence of local lifts
  giving rise to semi-balanced relative Selmer and dual Selmer groups
  is exactly what is needed for the proof of \cite[Theorem
  5.1]{fkp-invent} to go through to produce a lift $\rho$ satisfying
  all the claimed properties.
\end{proof}

\subsection{Trianguline lifts of $\glo$ representations of number
  fields} \label{s:lifts}

Before proving our main theorem we need the following:
\begin{lem} \label{l:invodd}%
  If $\br: \Gamma_F \to \gln(k)$ is a $\glo$ representation,
  then for every real place $v$ of $F$,
  \begin{equation}
    c(\br,v)  = \begin{cases}
    2m^2 -1 & \text{if $n = 2m$ is even,} \\
    2m^2 + 2m & \text{if $n = 2m + 1$ is odd.}
  \end{cases}
  \end{equation}
\end{lem}

\begin{proof}
  This is an elementary computation.
\end{proof}

We now restate and prove our main theorem.
\begin{thm} \label{t:triang} %
  Let $F$ be a number field and $\br:\Gamma_{F,\mc{S}} \to \gln(k)$ a
  $\glo$ representation. If $p \gg_n 0$ and $\br$ satisfies all the
  conditions in Assumption A then there exists a finite set of places
  $\mc{S}' \supset \mc{S}$ and a finite extension $E'$ of $E$ with
  ring of integers $\co'$ such that $\br$ lifts to a $\glo$
  representation $\rho: \Gamma_{F,\mc{S}'} \to \gln(\co')$ which is
  regular trianguline at all primes of $F$ above $p$. Furthermore, one
  can also ensure that $\rho(\Gamma_{F, \mc{S}'})$ contains an open
  subgroup of $\mr{SL}_n(\co')$.
\end{thm}

\begin{proof}
  The proof is an application of Theorem \ref{t:general}.

  We begin by choosing any lift $\mu$ of $\bar{\mu}$. The conditions
  of Assumption A hold tautologically, so we only need to show
  that the conditions of Assumption B also hold.

  Since the group $G$ here is $\gln$, for $v \nmid p$ this holds by
  \cite[\S 2.4.4]{clozel-harris-taylor} with $d_v = \dim_k(\fgder)$.
  The main change that we need to make is in the local arguments for
  primes $v \mid p$. For each such prime, we apply Lemma
  \ref{l:triang} (with the $F$ there being $F_v$) to obtain a set
  $U_v$ of trianguline lifts of $\br|_{\Gamma_{F_v}}$. For these $U_v$
  we have by (2) of Lemma \ref{l:triang}
  $d_v = n^2 -1 + [F_v:\Q_p]\bigl (\tfrac{n(n+1)}{2} - 1 \bigr )$.

  Since $\sum_{v \mid p} [F_v:\Q_p] = [F:\Q]$,  we get
  $
    \sum_{v \mid p} d_v = \sum_{v \mid p} (n^2 - 1) +[F:\Q] \bigl
    (\tfrac{n(n+1)}{2} - 1 \bigr ).
  $
  On the other hand
  $ \sum_{v | \infty} c(\br, v) = \sum_{v \ \mr{real}} c(\br, v) +
  \sum_{v \ \mr{complex}} c(\br, v) $. Using Lemma \ref{l:invodd} we
  see that if $v$ is a real place then
  $c(\br,v) \leq \bigl (\tfrac{n(n+1)}{2} - 1 \bigr ) $ and if $v$ is
  a complex place then
  $c(\br,v) = n^2 - 1 \leq 2 \bigl (\tfrac{n(n+1)}{2} - 1 \bigr )$.
  Using this it follows that Equation \eqref{e:ad} holds, so the proof
  is complete.
\end{proof}

\begin{rem} %
  It seems reasonable to expect that Assumption A and the assumption
  that $p \gg_n 0$ are superfluous and trianguline lifts as in Theorem
  \ref{t:triang} exist as long as $\br$ is $\glo$.
\end{rem}

\begin{rem} \label{r:upper}%
  If $\br|_{\Gamma_{F_v}}$ is upper triangular for a prime $v$ above
  $p$, using Remark \ref{r:ord} one may arrange that the same holds
  for $\rho|_{\Gamma_{F_v}}$.
\end{rem}

\begin{rem}
  The proof of Theorem \ref{t:triang} actually gives a family of
  (fixed-determinant) lifts of $\br$ (unramified outside a fixed
  finite set $\mc{S}'$ and trianguline at all primes above $p$)
  parametrised by a
  $\bigl( [F:\Q] \bigl (\tfrac{n(n+1)}{2} - 1 \bigr ) - \sum_{v |
    \infty} c(\br,v) \bigr )$-dimensional $E$-adic manifold.
\end{rem}

\begin{cor} \label{c:triang}%
  Let $F$ be an arbitrary number field and
  $\br:\Gamma_{F,\mc{S}} \to \gln(k)$ a $\glo$ representation. If
  $p \gg_n 0$ and $\br|_{\Gamma_{F(\zeta_p)}}$ is absolutely
  irreducible, then there exists a finite set of places
  $\mc{S}' \supset \mc{S}$ and a finite extension $E'$ of $E$ with
  ring of integers $\co'$ such that $\br$ lifts to a $\glo$
  representation $\rho: \Gamma_{F,\mc{S}'} \to \gln(\co')$ which is
  regular trianguline at all primes of $F$ above $p$. Furthermore, one
  can also ensure that $\rho(\Gamma_F)$ contains an open subgroup of
  $\mr{SL}_n(\co')$.
\end{cor}
\begin{proof}
  This follows immediately from Theorem \ref{t:triang} once we note
  that the irreducibility assumption implies that (3), (4) and (5) of
  Assumption A automatically hold; see \cite[Corollary
  A.7]{fkp-duke} for a more general statement that holds for arbitrary
  $G$.
\end{proof}

\begin{rem}
  The recent results of \cite{bop-local} show that the second bulleted
  assumption in \cite[Theorem 6.21]{fkp-duke} is always satisfied for
  $G = \gln$. Consequently, if $F$ is an arbitrary number field, and
  we drop the $\glo$ assumption in Thereom \ref{t:triang}, we can
  construct characteristic zero lifts which are unramified outside a
  finite set of primes, but without any control at primes dividing $p$.
\end{rem}

\begin{rem}
  As far as we are aware, it is not known whether or not every $\glo$
  representation has a geometric lift, even when $F$ is $\Q$---the
  examples of Calegari mentioned in the introduction are not
  $\glo$---and this question is closely related to the question
  whether torsion cohomology Hecke eigenclasses of arithmetic subgroups of $\gln$
  lift to characteristic zero after passing to a finite index
  subgroup. However, Conjecture 1.2.3 of \cite{hansen-eig} combined
  with Theorem \ref{t:triang} implies that any irreducible
  representation $\br$ satisfying the hypotheses of the theorem is the
  reduction of the Galois representation associated to the cohomology
  of an arithmetic subgroup of $\gln$ with coefficients in a suitable
  infinite-dimensional module (see \cite[\S 5]{joh-newt-ext}) which is
  a $\Q_p$-vector space.
\end{rem}

\subsection{Trianguline lifts for representations into connected
  reductive groups}

Using Lemma \ref{l:triang2} (and the results of \cite{fkp-duke} and
\cite{fkp-invent}) we now prove a generalisation of Theorem
\ref{t:triang} for representations valued in general split connected
reductive groups $G$ satisfying the assumptions of the lemma. Before
we can do this though we need an analogue of the definition of $\glo$
used for $\gln$. The definition below, motivated by the numerics of
the trianguline local condition, is weaker than $\glo$ when applied to
$\gln$, but it still suffices to prove our lifting result.
\begin{defn} \label{d:odd2}%
  $ $
  \begin{enumerate}
  \item Let $K$ be any field and $G$ a split reductive group (not
    necessarily connected) over $K$. An involution $c \in G(K)$ is
    said to be \emph{t-odd} if $\dim(\fgder)^c \leq \dim(B^{\der})$,
    where $B^{\der}$ is a Borel subgroup of $G^{\der}$.
  \item Let $F$ be a number field, $K$ any topological field,
    and $G$ as above. We say that a continuous representation
    $\br: \Gamma_F \to G(K)$ is \emph{t-odd} if for every real place
    $v$ of $F$ and $c_v \in \Gamma_{F_v}$ the corresponding complex
    conjugation, the involution $\br(c_v) \in G(K)$ is t-odd as in
    (1).
  \end{enumerate}
\end{defn}

We then have the following generalisation of Theorem \ref{t:triang}.
\begin{thm} \label{t:triang2} %
  Let $F$ be a number number field, $G$ a connected and split
  reductive group, and $\br:\Gamma_{F,\mc{S}} \to G(k)$ a t-odd
  continuous representation. Assume $p \gg_{G} 0$, $\br$ satisfies the
  conditions of Assumption A and $G$ has no factor of type $\mr{G}_2$,
  $\mr{F}_4$ or $\mr{E}_8$.  Fix a lift ${\mu}$ of $\bar{\mu}$ as in
  \ref{s:general} and also assume that after possibly replacing $E$ by
  a finite extension, for each finite prime $v$ of $F$,
  $\br|_{\Gamma_{F_v}}$ has a lift to $G(\co)$ with determinant
  $\mu|_{\Gamma_{F_v}}$ which is trianguline and very regular if
  $v \mid p$. Then there exists a finite set of places
  $\mc{S}' \supset \mc{S}$ and an extension $E'$ of $E$ with ring of
  integers $\co'$ such that $\br$ lifts to a t-odd representation
  $\rho: \Gamma_{F, \mc{S}'} \to G(\co')$ with determinant $\mu$ which
  is very regular trianguline at all primes of $F$ above
  $p$. Furthermore, one can also ensure that
  $\rho(\Gamma_{F, \mc{S}'})$ contains an open subgroup of
  $G^{\der}(\co')$.
\end{thm}

\begin{proof}
  The proof is essentially the same as that of Theorem \ref{t:triang}
  after noting that the definition of t-odd is designed precisely to
  ensure that the cocycles obtained by applying Proposition
  \ref{prop2} and Lemma \ref{l:triang2} for primes $v \mid p$ and
  \cite[Proposition 4.7]{fkp-duke} for other primes are adequate in
  the sense of Definition \ref{d:strong}. This ensures, by Lemma
  \ref{l:semibal}, that the (relative) Selmer and dual Selmer groups
  are semi-balanced which is the key condition needed for the
  arguments of \cite{fkp-duke} to go through as in the proof of
  Theorem \ref{t:triang}.
\end{proof}
It will be clear to the reader that using this theorem one may also
prove an analogue of Corollary \ref{c:triang}, but we do not formulate
it explicitly here. The $t$-odd assumption in the theorem can also be
replaced by the slightly weaker condition
$ \sum_{v \mid \infty} \dim(\fgder)^{\Gamma_{F_v}} \leq
[F:\Q]\dim(B^{\der})$.
\begin{rem}
  The existence of lifts for primes $v \nmid p$ has been proved for
  some groups other than $\gln$ by Booher \cite{booher:minimal} and
  very regular trianguline (in fact, crystalline) lifts for primes
  $v \mid p$ are known to exist when $\br|_{\Gamma_{F_v}}$ is
  absolutely irreducible by recent work of Lin \cite{lin-irred}.
\end{rem}

\bibliographystyle{siam}
\bibliography{biblio.bib}

\end{document}